\documentclass{article}
\usepackage{amsmath,amssymb,amsthm,amsfonts,latexsym,graphicx,subfigure}
\usepackage[affil-it]{authblk}
\usepackage[numbers,sort&compress]{natbib}
\usepackage{xcolor}

\usepackage{hyperref}

\theoremstyle{plain}
\newtheorem{theorem}{Theorem}[section]

\newtheorem{question}{Question}
\newtheorem{lemma}[theorem]{Lemma}

\theoremstyle{definition}
\newtheorem{definition}[theorem]{Definition}

\theoremstyle{remark}

\newcommand{\ex}{\mathrm{ex}}
\newcommand{\T}{\Tilde{T}}
\renewcommand{\t}{\Tilde{t}}
\newcommand{\K}[2]{K_{#1}^{(#2)}}
\renewcommand{\l}{\left}
\renewcommand{\r}{\right}

\usepackage[affil-it]{authblk}

\begin{document}
\bibliographystyle{plain} 
\title{On Turán numbers for disconnected hypergraphs}
\author[1,2]{Raffaella Mulas\thanks{r.mulas@vu.nl} }
\author[1,3]{Jiaxi Nie\thanks{jiaxi\_nie@fudan.edu} }
\date{}
\affil[1]{Max Planck Institute for Mathematics in the Sciences, Leipzig, Germany}
\affil[2]{Vrije Universiteit Amsterdam, Amsterdam, The Netherlands}
\affil[3]{Shanghai Center for Mathematical Sciences, Fudan University, Shanghai, China}
\maketitle
\allowdisplaybreaks[4]

\begin{abstract}
We introduce the following simpler variant of the Turán problem:
Given integers $n>k>r\geq 2$ and $m\geq 1$, what is the smallest integer $t$ for which there exists an $r$-uniform hypergraph with $n$ vertices, $t$ edges and $m$ connected components such that any $k$-subset of the vertex set contains at least one edge? We prove some general estimates for this quantity and for its limit, normalized by $\binom{n}{r}$, as $n\rightarrow \infty$. Moreover, we give a complete solution of the problem for the particular case when $k=5$, $r=3$ and $m\geq 2$.

\vspace{0.2cm}
\end{abstract}


\section{Introduction}
Given $r\geq 2$, an \emph{$r$-uniform hypergraph}, or \emph{$r$-graph} for short, is a pair $H=(V,E)$, where $V=V(H)$ is a finite set of \emph{nodes} or \emph{vertices}, and $E=E(H)$ is a set of $r$-subsets of $V$, called \emph{edges}. In particular, a graph is a $2$-graph. Given an $r$-graph $F$, the \emph{Turán number} $\ex(n,{F})$  is the largest integer $t$ such that there exists an $r$-graph on $n$ vertices and $t$ edges that does not contain $F$ as a sub-hypergraph.\newline
The \emph{Turán problem} consists of determining or estimating $\ex(n,K_k^{(r)})$, where $K_k^{(r)}$ is a \emph{complete} $r$-graph on $k$ vertices, i.e., the hypergraph consisting of all possible $r$-subsets of $V$. Moreover, one is also interested in estimating the limit
$$
\pi_{r,k}:=\lim_{n\to\infty}\frac{\ex(n,\K{k}{r})}{\binom{n}{r}},
$$
where $\binom{n}{r}$ is the number of edges of $\K{n}{r}$. Note that the edge density of an $n$-vertex $r$-graph equals the average edge density of its $(n-1)$-vertex induced subgraphs. Hence $\ex(n,\K{k}{r})/\binom{n}{r}$ is non-increasing, which implies that $\pi_{r,k}$ is well-defined. This simple averaging argument was first introduced by Katona, Nemetz and Simonovits~\cite{katona1964problem}.  \newline
The Tur\'an problem was introduced in 1941 by Paul Turán \cite{turan1941}, who showed that $\pi_{2,k}=1/(k-1)$. More specifically, let $\mathcal{T}_{n,k}$ be the graph on $n$ vertices such that the vertex set can be written as $V=V_1\sqcup\cdots\sqcup V_k$, each $V_i$ has size either $\lceil n/k\rceil$ or $\lfloor n/k \rfloor$, and two vertices form an edge if and only if they belong to different $V_i$'s. Tur\'an showed that
$$
\ex(n,\K{k}{2})=|E(\mathcal{T}_{n,k-1})|
$$
and $\mathcal{T}_{n,k-1}$ is the only $n$-vertex, $\K{k}{2}$-free graph that achieves
this number of edges. Eighty years later, there are still very few results on the Tur\'an's problem for $k>r\ge 3$. Tur\'an conjectured that $\pi_{3,4}=5/9$ and $\pi_{3,5}=3/4$. There are exponentially many constructions achieving the conjectured densities, see~\cite{brown1973extremal,fon1988method,frohmader2008more,kostochka1982class}. The best upper bound so far for $\pi_{3,4}$ is $0.561666$ by Razborov~\cite{razborov20103} using flag algebra method. Some classical surveys are \cite{keevash2011survey,sidorenko1995we,de1991current,furedi1991}, and other related results are presented, for instance, in \cite{SIDORENKO1997,NORIN2018,MA2020,Gishboliner,keevash2005,frankl1988extremal,mubayi2006hypergraph,bollobas1974three,keevash2004stability,de2000maximum,simonovits1968method,mubayi2005proof,pikhurko2010analytic}.\newline
Paul Erdős, who was a close collaborator of Turán, rarely offered prizes for problems that were posed by others \cite{chung1998erdos}, offered 500 dollars for determining 
$
\pi_{r,k}
$
for even one single pair $k>r>2$, and he offered 1000 dollars for solving the whole set of problems \cite{erdHos1981combinatorial}. After Erdős' death, his close collaborators Fan Chung and Ron Graham declared that they were willing to offer these prizes, as a way to honor him \cite{chung1998erdos}. At the time of writing, Fan Chung is currently in charge of these rewards.
\newline

The Turán problem can also be reformulated in a dual way, as follows.\newline 
Given an $r$-graph $H=(V,E)$, its \emph{complement} is the hypergraph $H^\textrm{c}:=(V,E^\textrm{c})$, where
    \begin{equation*}
        E^\textrm{c}:=\{f\subseteq V:|f|=r \text{ and }f\notin E \}.
    \end{equation*}Clearly,
    \begin{equation*}
        |E^\textrm{c}|=\binom{n}{r}-|E|,
    \end{equation*}therefore, given $t\in\mathbb{N}$,
    \begin{equation*}
        |E|\leq t \iff |E^\textrm{c}|\geq \binom{n}{r}-t.
    \end{equation*}Also, $H$ is $K_k^{(r)}$-free if and only if $H^\textrm{c}$ satisfies the condition
    \begin{equation}\label{complement-condition}
    \forall v_1,\ldots,v_k\in V \Longrightarrow \exists f\subseteq \{v_1,\ldots,v_k\}: f\in E^\textrm{c}.
    \end{equation}Hence, $H$ is an \emph{optimal solution} for the Turán problem, i.e., it maximizes the number of edges among all $K_k^{(r)}$-free $r$-graphs on $n$ nodes, if and only if $H^\textrm{c}$ minimizes the number of edges among all $r$-graphs on $n$ nodes that satisfy \eqref{complement-condition}. Motivated by this dual (and equivalent) formulation, we say that the \emph{dual Turán problem} consists of determining or estimating
    \begin{equation*}
        T(n,K_k^{(r)}):=\binom{n}{r}-\ex(n,K_k^{(r)}),
    \end{equation*}which is the smallest integer $t$ such that there exists an $r$-graph on $n$ vertices and $t$ edges satisfying \eqref{complement-condition}.
    Further, we let 
$$
t_{r,k}:=\lim_{n\to\infty}\frac{T(n,\K{k}{r})}{\binom{n}{r}}.
$$
Clearly, we have $t_{r,k}=1-\pi_{r,k}$.\newline

In this paper, we introduce and study a variant of the Tur\'an number concerning the number of connected components. Namely, we let $T(n,\K{k}{r};m)$ be the smallest integer $t$ such that there exists an $r$-graph with $n$ vertices, $t$ edges and $m$ connected components that satisfies \eqref{complement-condition}. We also consider the limit
$$
t_{r,k}(m):=\lim_{n\to\infty}\frac{T(n,\K{k}{r};m)}{\binom{n}{r}}.
$$
The existence of this limit can be proved by an argument similar to that of $\pi_{r,k}$. One of our main results is the following theorem, that we shall prove in Section \ref{section:connected}.

\begin{theorem}\label{thm:r-graph}
  Let $k>r\ge2$ be integers. If $n\ge k+\binom{k-2}{r-1}$ and $T(n,\K{k}{r})=T(n,\K{k}{r};m)$, then 
  $$
  m\le \l\lfloor\frac{k-1}{r-1}\r\rfloor.
  $$
\end{theorem}

In particular, for $k\le 2r-2$, Theorem~\ref{thm:r-graph} implies that the optimal solution of $T(n,\K{k}{r})$ must be connected when $n$ is large enough.\newline
 
Note that in the case of graphs, the optimal solutions of $T(n,\K{k}{2})$ always have $k-1$ connected components. That is, $T(n,\K{k}{2};k-1)<T(n,\K{k}{2};m)$ for any $m<k-1$. It is natural to ask whether this phenomenon extends to $r$-graphs, i.e., if the optimal solutions of $T(n,\K{k}{r})$ always have $\lfloor{(k-1)}/{(r-1)}\rfloor$ connected components. We show that the answer to this question is ``no'' for $r$-graphs when $k-1$ is a multiple of $r-1$:
\begin{theorem}\label{thm:disconnected_not_better}
For integers $r\ge3$, $k\ge 1$ and $n\ge (r-1)k+1+\binom{(r-1)k-1}{r-1}$ such that $k|n$, there exists an integer $m<k$ such that
$$
T(n,\K{(r-1)k+1}{r};k)\ge T(n,\K{(r-1)k+1}{r};m).
$$
\end{theorem}
Moreover, we determine $t_{3,2m+1}(m)$ and $t_{3,2m+2}(m)$ in terms of $t_{3,4}$:

\begin{theorem}\label{thm:disconnected_3-graph}
For integers $m\ge1$,
$$
t_{3,2m+1}(m)=\frac{1}{m^2},
$$
$$
t_{3,2m+2}(m)=(m-1+t_{3,4}^{-\frac{1}{2}})^{-2}.
$$
\end{theorem}
We point out that it is unclear whether a similar result also holds for $r\ge 4$. For example, we do not know if $t_{4,7}(2)=1/8$. This is because the current best lower bound for $t_{4,6}$, to the best of our knowledge, is $1/10$, which is smaller than $1/8$. \newline

 \textbf{Structure of the paper.} In Section \ref{section:general} we give further definitions and we prove some general results that will be needed throughout the paper. In Section \ref{section:Turan_3_5} we give a complete solution of the problem for the case when $k=5$, $r=3$ and $m\geq 2$. In Section \ref{section:connected} we prove Theorem \ref{thm:r-graph}, and in Section \ref{section:disconnected} we prove Theorem~\ref{thm:disconnected_not_better} and Theorem~\ref{thm:disconnected_3-graph}. Finally, in Section \ref{Section:Questions} we propose some open questions.

    \section{Basic definitions and general results}\label{section:general}
    
    In the introduction we gave several definitions regarding hypergraphs and the Turán problem. In this section, we give further definitions and we prove some general results that hold for any $k>r \geq 3$.
    
    \begin{definition}
    Given a hypergraph $H=(V,E)$ and a vertex $v\in V$, we let $H\setminus\{v\}:=(V\setminus \{v\},E\setminus\{v\})$, where $E\setminus\{v\}:=\{e\in E\,:\,v\notin e\}$.
    \end{definition}
    
      \begin{definition}
    Given an $r$-graph $H=(V,E)$, an $r$-subset of $V$ is a \emph{non-edge} of $H$ if it does not belong to $E$, or equivalently, if it belongs to $E^\textrm{c}$.
    \end{definition}

    Given an hypergraph $H$, we let $C(H)$ denote the number of its connected components.
    
    \begin{definition}Let $H=(V,E)$ be an hypergraph. An \emph{independent set} of $H$ is a subset of $V$ which does not contain any edge of $H$. The \emph{independence number} of $H$, denoted $\alpha(H)$, is the maximum size of an independent set.\newline
    The {\em independence sequence} of $H$, denoted $S(H)$, is the multiset of size $C(H)$ whose members are the independence numbers of each component.
    \end{definition}
    
 Given a multiset $S$ whose elements are positive integers, we let $|S|$ denote the number of entries in $S$, and we let $\|S\|$ denote the sum of all entries of $S$. That is, if $S=\{s_1,s_2,\ldots,s_t\}$, then $|S|:=t$ and $\|S\|:=s_1+s_2+\cdots+s_t$. \newline Moreover, given two multisets $S$ and $S'$, we let $S\uplus S'$ denote the union of $S$ and $S'$. For instance, $\{3,1\}\uplus\{2,2,1\}=\{3,2,2,1,1\}$. Given a positive integer $m$, we let $m\cdot S$ denote the multiset union of $m$ copies of $S$. For example, $3\cdot\{2,1\}=\{2,2,2,1,1,1\}$.\newline
 
We also let $\T(n,r;S)$ be the smallest number of edges in an $n$-vertex $r$-graph $H$ such that $S(H)=S$, and we let
$$
\t_r(S):=\lim_{n\to\infty}\frac{\T(n,r;S)}{\binom{n}{r}}.
$$
The existence of this limit can be proved by a simple averaging argument similar to that for $\pi_{r,k}$. When $S=\{s\}$, we write $\T(n,r;s)=\T(n,r;\{s\})$ and $\t_r(s)=\t_r(\{s\})$ for short. One can then check that 
$$
t_{r,k}(m)=\min_{S\,:\,|S|=m,\,\|S\|=k-1}\t_r(S).
$$

The following lemmas will be needed throughout the paper. 
\begin{lemma}\label{lemma:concatenate}
Let $S^{(1)}$, $S^{(2)}$ and $S$ be multisets. If $$\t_r(S^{(1)})\ge \t_r(S^{(2)}),$$
then 
$$\t_r(S^{(1)}\uplus S)\ge \t_r(S^{(2)}\uplus S).$$
\end{lemma}

\begin{proof}
It suffices to show that $\T(n,r;S^{(1)}\uplus S)\ge \T(n,r;S^{(2)}\uplus S)+o(n^r)$. For any optimal solution $H$ of $\T(n,r;S^{(1)}\uplus S)$, let $H_1$ be the induced subgraph of $H$ such that $S(H_1)=S^{(1)}$. Since 
$$
\T(n,r;S^{(1)})\ge \T(n,r;S^{(2)})+o(n^r),
$$ there is an $r$-graph $H_2$ on the same vertex set as $H_1$ such that $S(H_2)=S^{(2)}$ and $e(H_1)\ge e(H_2)+o(n^r)$. Let $H'$ be the $r$-graph obtained from $H$ by replacing $H_1$ with $H_2$. Then, we have $S({H'})=S^{(2)}\uplus S$ and $e(H')\le e(H)+o(n^r)$. Therefore, $\T(n,r;S^{(1)}\uplus S)\ge \T(n,r;S^{(2)}\uplus S)+o(n^r)$.
\end{proof}

\begin{lemma}\label{lemma:binom}
For integers $b\geq a>r>1$,
\begin{equation*}
    \binom{a}{r}+\binom{b}{r}<\binom{a-1}{r}+\binom{b+1}{r}.
\end{equation*}
\end{lemma}
\begin{proof}
  We use the fact that
  \begin{equation*}
      \binom{a}{r}= \binom{a-1}{r}+ \binom{a-1}{r-1}
  \end{equation*}and, similarly,
  \begin{equation*}
      \binom{b+1}{r}= \binom{b}{r}+ \binom{b}{r-1}.
  \end{equation*}This implies that
  \begin{align*}
      \binom{a}{r}+\binom{b}{r}<\binom{a-1}{r}+\binom{b+1}{r} & \iff  \binom{a-1}{r-1}<\binom{b}{r-1}, 
  \end{align*}which is true since $b>a-1$.
\end{proof}

\begin{lemma}\label{Lemma:weighted_Jensen}
For $i\in\{1,\ldots,m\}$, let $t_i>0$, and $p_i\in(0,1]$ be real numbers such that $\sum_{i=1}^m p_i=1$. Then, for any integer $r\ge 2$, \begin{equation}\label{equation:weighted_Jensen}
    \sum_{i=1}^mt_i\cdot p_i^r\ge\l(\sum_{i=1}^mt_i^{-\frac{1}{r-1}}\r)^{-r+1}. 
\end{equation}
Equality holds if and only if $p_i=t_i^{-\frac{1}{r-1}}/\sum_{j=1}^mt_j^{-\frac{1}{r-1}}$.
\end{lemma}

\begin{proof}
Let $c_i:=t_i^{-\frac{1}{r-1}}$, and let $C:=\sum_{i=1}^mc_i$. Then,
$$
(*):=\sum_{i=1}^mt_i\cdot p_i^r=\sum_{i=1}^mc_i\cdot\l(\frac{p_i}{c_i}\r)^r=C\sum_{i=1}^m\frac{c_i}{C}\cdot\l(\frac{p_i}{c_i}\r)^r.
$$
Hence, by Jensen's inequality, we have
$$
(*)\ge C\l(\sum_{i=1}^m\frac{c_i}{C}\cdot\frac{p_i}{c_i}\r)^r=C^{-r+1},
$$
and equality holds if and only if $$\frac{p_1}{c_1}=\frac{p_2}{c_2}=\cdots=\frac{p_m}{c_m}.$$ Together with $\sum_{i=1}^m p_i=1$, this implies that $p_i=c_i/C$ for all $i\in\{1,\ldots,m\}$.
\end{proof}

The next lemma establishes an equation for $\t_r(S)$.

\begin{lemma}\label{lemma:decompose}
Let $r\ge3$ be an integer, and let $S=\{s_1,s_2,\dots,s_m\}$ be a multiset such that $s_i\ge r-1$ for all $i\in\{1,\ldots,m\}$. Then,
\begin{equation*}\label{equation:decompose}
    \t_r(S)=\l(\sum_{i=1}^m\t_r(s_i)^{-\frac{1}{r-1}}\r)^{-r+1}.
\end{equation*}
\end{lemma}

\begin{proof}
Let $G_n$ be an $r$-graph on $n$ vertices whose independence sequence is $S$ such that $e(G_n)=\T(n,r;S)$, and let $n_i$ be the number of vertices in the component of $G_n$ corresponding to $s_i$, for $1\le i\le m$. Then, by definition,
$$
\frac{e(G_n)}{\binom{n}{r}}\ge\sum_{i=1}^m\frac{\T(n_i,r;s_i)}{\binom{n}{r}}= \sum_{i=1}^m\frac{\binom{n_i}{r}}{\binom{n}{r}}\frac{\T(n_i,r;s_i)}{\binom{n_i}{r}} 
$$
It is not hard to check that
$$
\liminf_{n\to\infty}\frac{\binom{n_i}{r}}{\binom{n}{r}}\frac{\T(n_i,r;s_i)}{\binom{n_i}{r}}=\liminf_{n\to\infty}\l(\frac{n_i}{n}\r)^r\t_r(s_i);
$$
when $\liminf(n_i/n)>0$, this is clearly true; and when $\liminf(n_i/n)=0$, both sides are zero. Hence, by Lemma~\ref{Lemma:weighted_Jensen},

$$
\t_r(S)=\liminf_{n\to\infty}\frac{e(G_n)}{\binom{n}{r}}\ge\liminf_{n\to\infty}\sum^m_{i=1}\l(\frac{n_i}{n}\r)\t_r(s_i)\ge \l(\sum_{i=1}^m\t_r(s_i)^{-\frac{1}{r-1}}\r)^{-r+1}.
$$

On the other hand, we can construct a sequence of $r$-graphs whose density converges to $\l(\sum_{i=1}^m\t_r(s_i)^{-\frac{1}{r-1}}\r)^{-r+1}$ by taking the union of the optimal solutions of $\T(\lfloor np_i\rfloor,r;{s_i})$, for $1\le i\le m$, where 
$$
p_i=\frac{\t_r(s_i)^{-\frac{1}{r-1}}}{\sum_{j=1}^m\t_r(s_j)^{-\frac{1}{r-1}}}.
$$ We can therefore conclude that $\t_r(S)= \l(\sum_{i=1}^m\t_r(s_i)^{-\frac{1}{r-1}}\r)^{-r+1}.$
\end{proof}

    \section{Solution for $r=3$, $k=5$ and $m\geq 2$}\label{section:Turan_3_5}
    
        In this section we solve the problem for the case when $r=3$, $k=5$, and $m\geq 2$. This  will serve as a motivation for our next results. 

    \begin{lemma}\label{lem:components}
    Let $H=(V,E)$ be a $K_5^{(3)}$-free $3$-graph on $n\geq 6$ nodes. Then, $H^\textrm{c}$ has at most three connected components. Moreover,
    \begin{itemize}
        \item If $H^\textrm{c}$ has three connected components, then $H^\textrm{c}$ is given by two isolated vertices together with a complete $3$-graph on $n-2$ nodes.
        \item If $H^\textrm{c}$ has two connected components and it has no isolated vertices, then each connected component of $H^\textrm{c}$ is a complete $3$-graph.
        \item If $H^\textrm{c}$ has two connected components and one of them is an isolated vertex, then the other one is $K_4^{(3)}$-free in $H$.
    \end{itemize}
    \end{lemma}
    \begin{proof}
    We first assume, for the sake of a contradiction, that $H^\textrm{c}$ has at least four connected components $V_1,\ldots,V_4$. Let $v_1\in V_1,\ldots, v_4\in V_4$, and let $v_5\in V$ belong to any of the connected components. Since $H$ is $K_5^{(3)}$-free, $H^\textrm{c}$ satisfies \eqref{complement-condition} and therefore there exists $f\subseteq\{v_1,\ldots,v_5\}$ such that $f\in E^\textrm{c}$. Since $|f|=3$ and since $V_1,\ldots,V_4$ are not joined by any edge, this gives a contradiction. Hence, $H^\textrm{c}$ has at most three connected components.\newline
    
    If $H^\textrm{c}$ has three connected components, then clearly one component has at least one edge and therefore at least $3$ vertices.
    Assume, for the sake of a contradiction, that $H^\textrm{c}$ has three connected components $V_1,V_2,V_3$ and at least two of them, say $V_1$ and $V_2$, have cardinality larger than $1$. Then, we can pick $v_1,v_2\in V_1$, $v_3,v_4\in V_2$ and $v_5\in V_3$. By \eqref{complement-condition}, there exists $f'\subseteq\{v_1,\ldots,v_5\}$ such that $f'\in E^\textrm{c}$. Since $f'$ has cardinality $3$, this gives a contradiction. Therefore, if $H^\textrm{c}$ has three connected component, then two of them are given by isolated vertices. Now, assume that $H^\textrm{c}$ has three connected components $V_1,V_2,V_3$, where $V_2$ and $V_3$ are given by isolated vertices. Given $v_1,v_2,v_3\in V_1$, $v_4\in V_2$, $v_5\in V_3$, \eqref{complement-condition} implies that $\{v_1,v_2,v_3\}\in E^\textrm{c}$. Since this holds for all $v_1,v_2,v_3\in V_1$, $V_1$ is a complete $3$-graph on $n-2$ nodes.\newline
    
    Now, assume that $H^\textrm{c}$ has two connected components $V_1,V_2$, and it has no isolated vertices. Since edges have cardinality $3$, there cannot be connected components of cardinality $2$, therefore $|V_i|\geq 3$ for $i=1,2$. Let $v_1,v_2,v_3\in V_1$, and $v_4,v_5\in V_2$. By \eqref{complement-condition}, there exists $f''\subseteq\{v_1,\ldots,v_5\}$ such that $f''\in E^\textrm{c}$. Since there are no edges between $V_1$ and $V_2$, this implies that $f=\{v_1,v_2,v_3\}$. This is true for all $v_1,v_2,v_3\in V_1$, and the same reasoning can be applied to $V_2$. Hence, both $V_1$ and $V_2$ are complete $3$-graphs.\newline
    
    Finally, assume that $H^\textrm{c}$ has two connected components and one of them is an isolated vertex $v_1$. Let $v_2,\ldots,v_5\in V\setminus\{v_1\}$. If $v_2,\ldots,v_5$ form a complete $3$-graph on $4$ vertices in $H$, then $v_1,\ldots,v_5$ form a complete $3$-graph on $5$ vertices in $H$, which leads to a contradiction. Hence, $V\setminus\{v_1\}$ is $K_4^{(3)}$-free in $H$.
    \end{proof}
    Lemma \ref{lem:components} allows us to prove the following
    
        \begin{theorem}\label{thm:max}
            Let $H=(V,E)$ be a $K_5^{(3)}$-free $3$-graph on $n\geq 6$ nodes such that $H^\textrm{c}$ has more than one connected component.
            \begin{itemize} 
            \item If $n=2m$, then 
            \begin{equation*}
              |E^\textrm{c}|\geq 2\cdot \binom{m}{3}, 
            \end{equation*}or equivalently
            \begin{equation*}
              |E|\leq m^2(m-1),
            \end{equation*}
            with equality if and only if $H^\textrm{c}$ is the disjoint union of two copies of $K_m^{(3)}$.\newline
            \item If $n=2m+1$, then
            \begin{equation*}
                |E^\textrm{c}|\geq \binom{m}{3}+\binom{m+1}{3},
            \end{equation*}or equivalently
             \begin{equation*}
                |E|\leq m^3+\frac{m^2}{2}-\frac{m}{2},
            \end{equation*}
            with equality if and only if $H^\textrm{c}$ is the disjoint union of $K_m^{(3)}$ and $K_{m+1}^{(3)}$.
            \end{itemize}
        \end{theorem}
        \begin{proof}We consider three cases.
        \begin{enumerate}
        \item Case 1: $H^\textrm{c}$ has more than two connected components. By Lemma \ref{lem:components}, $H^\textrm{c}$ is given by two isolated vertices together with a complete $3$-graph on $n-2$ nodes. Therefore, 
        \begin{equation*}
        |E^\textrm{c}|=\binom{n-2}{3}.
        \end{equation*}
         It is easy to check that, for $n=2m$,
         \begin{equation*}
             |E^\textrm{c}|= \binom{n-2}{3}>2\cdot \binom{m}{3}
         \end{equation*}and, for $n=2m+1$,
         \begin{equation*}
        |E^\textrm{c}|=  \binom{n-2}{3}>\binom{m}{3}+\binom{m+1}{3}.
         \end{equation*}
            \item Case 2: $H^\textrm{c}$ has two connected components and no isolated vertices. By Lemma \ref{lem:components}, each connected component of $H^\textrm{c}$ is a complete $3$-graph. Hence, if they have cardinality $c_1$ and $c_2$, respectively, then
        \begin{equation*}
        |E^\textrm{c}|=\binom{c_1}{3}+\binom{c_2}{3}.
        \end{equation*}By Lemma~\ref{lemma:binom}, for $n=2m$ this is minimized precisely by $c_1=c_2=m$, while for $n=2m+1$ this is minimized by $c_1=m$ and $c_2=m+1$. 
        \item Case 3: $H^\textrm{c}$ has two connected components and one of them is an isolated vertex $\hat{v}$. By Lemma \ref{lem:components}, $H\setminus\{\hat{v}\}$ is $K_4^{(3)}$-free. Therefore, using the bound $\ex(n-1,K^{(3)}_4)\le (0.63)\cdot \binom{n-1}{3}$ in \cite{de1991current,chung-upper},
        \begin{align*}
            |E|&=\deg \hat{v}+|E(H\setminus\{\hat{v}\})|\\
            &\leq \binom{n-1}{2}+\ex(n-1,K^{(3)}_4)\\
            &<\binom{n-1}{2}+(0.63)\cdot \binom{n-1}{3}.
        \end{align*}If $n=2m$, then one can check that
        \begin{equation*}
 |E|<\binom{2m-1}{2}+(0.63)\cdot \binom{2m-1}{3}<m^2(m-1).
        \end{equation*}Similarly, if $n=2m+1$, then
        \begin{equation*}
            |E|<\binom{2m}{2}+(0.63)\cdot \binom{2m}{3}<m^3+\frac{m^2}{2}-\frac{m}{2}.
        \end{equation*}
        \end{enumerate}This proves the claim.
        \end{proof}
   
    Note that the last result implies that $3/4=t_{3,5}(2)\le t_{3,5}(m)$ for $m\ge 2$.\newline
    Moreover, Turán conjectured that the disjoint unions of two complete $r$-graphs are optimal for $r=3$ and $k=5$ (in the general case when $m\geq 1$). However, as discussed in \cite{sidorenko1995we}, counterexamples are known for any odd $n\geq 9$. For $n$ even, no counterexample was found yet, and the conjecture has proved to be true up to $n=12$.

\section{Number of connected components}\label{section:connected}
In this section we prove Theorem~\ref{thm:r-graph}. As a preliminary result, we first prove a theorem showing that the optimal solutions of $T(n,\K{k}{r})$ have no isolated vertices, if $n$ is big enough.

\begin{lemma}\label{lem:isolated}
For integers $k> r\ge2$ and $n\geq k+\binom{k-2}{r-1}$, let $H=(V,E)$ be a $\K{k}{r}$-free $r$-graph with $n$ vertices and $\ex(n,\K{k}{r})$ edges. Then, $H^\textrm{c}$ has no isolated vertices. 
\end{lemma}

\begin{proof}
Suppose, for the sake of a contradiction, that $H^\textrm{c}$ contains an isolated vertex $u$, and let $H':=H\setminus\{u\}$. Then, $H'$ is a $\K{k-1}{r}$-free $r$-graph on $n-1$ vertices. Moreover, since $H$ is an optimal solution for the Turán problem, if any edge is added to $H$, then $H$ is not $\K{k}{r}$-free anymore. Therefore, if any edge is added to $H'$, then $H'$ is not $\K{k-1}{r}$-free anymore. Hence, given $f=\{v_1,\ldots,v_r\}\in E({H'}^{\textrm{c}})$, there exist $v_{r+1},\ldots,v_{k-1}\in V\setminus \{u\}$ such that there are no other edges among $v_1,\ldots,v_{k-1}$ in ${H'}^{\textrm{c}}$ other than $f$. Therefore, in particular, there are no edges in ${H'}^{\textrm{c}}$ among $v_2,\ldots,v_{k-1}$. Let now $w_1,\ldots,w_l\in V\setminus\{u,v_2,\ldots,v_{k-1}\}$, for some $l\geq 1$. Then, for each $i\in \{1,\ldots,l\}$, there exists $e_i\in E({H'}^{\textrm{c}})$ such that $e_i\subset \{w_i,v_2,\ldots,v_{k-1}\}$ and $w_i\in e_i$.\newline
If $l>\binom{k-2}{r-1}$, which is possible since $n\geq k+\binom{k-2}{r-1}$, there exist $i,j\in \{1,\ldots,l\}$ such that $i\neq j$ and $e_i\setminus\{w_i\}=e_j\setminus\{w_j\}=:U$.\newline
Let now $H''$ be the $r$-graph obtained from $H$ by deleting $U\cup\{u\}$ and adding $U\cup \{w_i\}$ and $U\cup\{w_j\}$ as edges. Then, $H''$ is $\K{k}{r}$-free. In fact, 
\begin{itemize}
    \item Any $k$-set not containing $U$ is not $\K{k}{r}$ since $H$ is $\K{k}{r}$-free;
    \item Any $k$-set containing $U\cup\{u\}$ contains the non-edge $U\cup\{u\}$;
    \item For any $k$-set $W$ containing $U$ but not $\{u\}$, fix a vertex $v\in U$. Note that $W\setminus\{v\}$ is a $(k-1)$-set .  This implies that there exists a non-edge of $H'$ contained in $W\setminus\{v\}$, which is neither $U\cup \{w_i\}$ nor $U\cup\{w_j\}$, and which is also a non-edge of $H''$.
\end{itemize}
   Therefore, $H''$ is an $\K{k}{r}$-free $r$-graph with $n$ nodes and more edges than $H$, which is a contradiction.
\end{proof}

Now we are ready to prove Theorem~\ref{thm:r-graph}.
\begin{proof}[Proof of Theorem~\ref{thm:r-graph}]
Let $H$ be an optimal solution of $T(n,\K{k}{r})$. By Lemma~\ref{lem:isolated}, $H$ has no isolated verteces. Hence, the independence number of each connected component of $H$ is at least $r-1$. Suppose, for the sake of a contradiction, that $m>\left\lfloor{(k-1)}/{(r-1)}\right\rfloor$. Then $\alpha(H)\ge(r-1)m>k-1$, which contradicts to $\alpha(H)\le k-1$.
\end{proof}

\section{Tur\'an numbers for disconnected $3$-graphs}\label{section:disconnected}
This section is dedicated to the proofs of Theorem~\ref{thm:disconnected_not_better} and Theorem~\ref{thm:disconnected_3-graph}. 
\begin{proof}[Proof of Theorem~\ref{thm:disconnected_not_better}]
The proof is divided into two cases. 

    \noindent Case 1: The optimal solution of $T(n,\K{(r-1)k+1}{r};k)$ has an isolated vertex. In this case, by Lemma~\ref{lem:isolated}, we have
    $$T(n,\K{(r-1)k+1}{r};k)> T(n,\K{(r-1)k+1}{r};k-1).$$
    
    \noindent Case 2: The optimal solution of $T(n,\K{(r-1)k+1}{r};k)$ does not contain isolated vertices. In this case, the $k$ components are all complete $r$-graphs. Hence, by Lemma~\ref{lemma:binom}, these components all have size $n/k$, implying that $$T(n,\K{(r-1)k+1}{r};k)=k\binom{n/k}{r}.$$ But we also know that $T(n,\K{(r-1)k+1}{r};1)\le k\binom {n/k}{r}$, by the construction of Sidorenko~(Construction 4 in Section 4.2 of \cite{sidorenko1995we}).
\end{proof}

For the proof of Theorem~\ref{thm:disconnected_3-graph}, we need the following key lemma.
\begin{lemma}\label{lemma:average}
For any integer $l\ge 2$,
$$
\t_3(\{l+1\}\uplus(l-1)\cdot\{1\})> \t_3(l\cdot\{2\}),
$$
$$
\t_3(\{l+2\}\uplus(l-1)\cdot\{1\})> \t_3(\{3\}\uplus(l-1)\cdot\{2\}).
$$
\end{lemma}

\begin{proof}
Due to de Caen~\cite{deCaen-extension}, we have 
$$
t_{r,l+1}\ge \frac{1}{\binom{l}{r-1}}.
$$
On the other hand, the construction of Sidorenko~(Construction 4 in Section 4.2 of \cite{sidorenko1995we}) gives
$$
\t_r(l)\le \frac{(r-1)^2}{l^2}.
$$
Hence for any integer $l\ge 2$,
\begin{equation}\label{equation:general}
    \frac{1}{\binom{l}{2}}\le t_{3,l+1}\le \t_3(l)\le \frac{4}{l^2}.
\end{equation}

Therefore, by definition and inequality~(\ref{equation:general})
$$
\t_3(\{l+1\}\uplus(l-1)\cdot\{1\})=\t_3(l+1)\ge \frac{1}{\binom{l+1}{2}},
$$
$$
\t_3(\{l+2\}\uplus(l-1)\cdot\{1\})=\t_3(l+2)\ge \frac{1}{\binom{l+2}{2}}.
$$
On the other hand, by Lemma~\ref{lemma:decompose} and inequality~(\ref{equation:general})
$$
\t_3(l\cdot\{2\})=l^{-2}<\frac{1}{\binom{l+1}{2}},
$$
$$
\t_3(\{3\}\uplus(l-1)\cdot\{2\})=(l-1+(\t_{3}(3))^{-\frac{1}{2}})^{-2}\le \left(\frac{2l+1}{2}\right)^{-2}<\frac{1}{\binom{l+2}{2}}.
$$
\end{proof}

We are now ready to prove Theorem~\ref{thm:disconnected_3-graph}.
\begin{proof}[Proof of Theorem~\ref{thm:disconnected_3-graph}]
For any multiset $S$ such that $|S|=m$ and $\|S\|=2m$, if there exists an element of $S$ which is larger than 2, then there exists an integer $l>2$ such that $\{l\}\uplus(l-2)\cdot\{1\}$ is a subset of $S$. Let $S'$ be a multiset obtained from $S$ by replacing $\{l\}\uplus(l-2)\cdot\{1\}$ by $(l-1)\cdot\{2\}$. Then, by Lemma~\ref{lemma:concatenate} and Lemma~\ref{lemma:average}, $\t_3(S')< \t_3(S)$. Also, note that the number of $2$'s in $S$ is smaller than that in $S'$. Hence, by repeating the above argument finitely many times and Lemma~\ref{lemma:decompose}, we obtain $\t_3(S)\ge \t_3(m\cdot\{2\})=1/m^2$. Therefore, 
$$
t_{3,2m+1}(m)=\min_{S\,:\,|S|=m,\,\|S\|=2m}\t_{3}(S)=\t_3(m\cdot\{2\})=\frac{1}{m^2}.
$$
Similarly, we have 
$$
t_{3,2m+2}(m)=\min_{S\,:\,|S|=m,\,\|S\|=2m+1}\t_{3}(S)=\t_3(\{3\}\uplus(m-1)\cdot\{2\})=(m-1+\t_{3}(3)^{-\frac{1}{2}})^{-2}.
$$
By Theorem~\ref{thm:r-graph}, the optimal solution of $T(n,\K{4}{3})$ must be connected when $n$ is large enough. This implies $t_{3,4}=\t_3(3)$. Therefore,
$$
t_{3,2m+2}(m)=(m-1+t_{3,4}^{-\frac{1}{2}})^{-2}.
$$
\end{proof}

\section{Open questions}\label{Section:Questions}

We conclude by formulating some open questions.

\begin{question}
Can we improve the bound $n\ge k+\binom{k-2}{r-1}$ in Theorem~\ref{thm:r-graph}?
\end{question}

\begin{question}
For $m<k$, can we explicitly express $t_{3,2k+1}(m)$ or $t_{3,2k+2}(m)$ in terms of the $t_{3,l}$'s?
\end{question}

\begin{question}
Can we prove results similar to Theorem~\ref{thm:disconnected_3-graph} for $r\ge 4$?
\end{question}

\section*{Funding}
Raffaella Mulas was supported by the Max Planck Society's Minerva Grant.

\section*{Acknowledgments}
The authors are grateful to the anonymous referees for the comments and suggestions that have greatly improved the first version of this paper.

\bibliography{Turan}

\end{document}